\documentclass{amsart}
\usepackage{ifthen}
\usepackage{verbatim}
\usepackage{amssymb,amsmath}
\usepackage{enumerate}
\usepackage[usenames,dvipsnames,svgnames,table]{xcolor}
\usepackage[colorlinks=true]{hyperref}

\usepackage{etoolbox}
\AtBeginEnvironment{proof}{\vspace{-14pt}}

\numberwithin{equation}{section}

\def\phi{\varphi}
\def\R{\mathbb R}
\def\N{\mathbb N}

\newcommand{\cm}{\mathcal{M}}

\newcommand{\bfM}{\mathbf{M}}
\newcommand{\bfH}{\mathbf{H}}

\newcommand{\eps}{\varepsilon}
\newcommand{\ra}{\rightarrow}
\newcommand{\be}{\begin{equation}}
\newcommand{\ee}{\end{equation}}

\newcommand{\res}{\ensuremath{\,\textsf{\small \upshape L}\,}}

\newcommand{\spt}{{\rm spt}\,}

\newtheorem{theorem}{Theorem}[section]

\newtheorem{proposition}[theorem]{Proposition}
\theoremstyle{definition}

\theoremstyle{remark}
\newtheorem{remark}[theorem]{Remark}
\numberwithin{equation}{section}
\def\uclhome{@ucl.ac.uk}

\begin{document}

\title{Consequences of strong stability of minimal submanifolds}
\author{Jason D. Lotay}
\author{Felix Schulze}
\address{
  Department of Mathematics, University College London, 25 Gordon St,
  London WC1E 6BT, UK}
\email{j.lotay\uclhome, f.schulze\uclhome}



\subjclass[2000]{}

\dedicatory{}


\begin{abstract} In this note we show that the recent dynamical stability result for small $C^1$-perturbations of strongly stable minimal submanifolds of C.-J.~Tsai and M.-T.~Wang \cite{TsaiWang17} directly extends to the enhanced Brakke flows of Ilmanen \cite{Ilmanen}.  We illustrate applications of this result, including a local uniqueness statement for strongly stable minimal  submanifolds amongst stationary varifolds, and a mechanism to flow through some singularities of Lagrangian mean curvature flow which are proved to occur by Neves \cite{NevesFTS}.
\end{abstract}

\maketitle

\section{Introduction}

Recently, Tsai and Wang \cite{TsaiWang17} considered $n$-dimensional minimal submanifolds $\Sigma\subset M$ where $(M,g)$ is an $(n+m)$-dimensional ambient Riemannian manifold. They
consider the partial Ricci operator on the normal bundle $N\Sigma$:
$$\mathcal{R}(V) = \text{tr}_{\Sigma}(R(\cdot,V)\cdot)^\perp,$$
where $R$ is the Riemann curvature tensor of $(M,g)$. They call $\Sigma$ \emph{strongly stable}  if $\mathcal{R} - \mathcal{A}$ is a (pointwise) positive operator on $N\Sigma$, where $\mathcal{A}$ is a quadratic expression in the second fundamental form of $\Sigma$ in $(M,g)$.  In coordinates this condition is equivalent to asking that there exists a constant $c_0>0$ such that, for any $p \in \Sigma$:
$$ -\sum_{\alpha,\beta,i}R_{i\alpha i\beta}v^\alpha v^\beta - \sum_{\alpha,\beta,i,j}h_{\alpha ij}h_{\beta ij} v^\alpha v^\beta \geq c_0 \sum_{\alpha} (v^\alpha)^2\, $$
for any $V = \sum_{\alpha} v^\alpha\bar{e}_\alpha \in N_p\Sigma$, where $(e_i)_{i=1,\ldots,n}$ and $(\bar{e}_\alpha)_{\alpha=1,\ldots,m}$ are orthonormal bases of $T_p\Sigma$ and $N_p\Sigma$ respectively and $(h_{\alpha ij})$ are the coefficients of $\mathcal{A}$. Note that strong stability implies the integrand in the second variation formula for the volume functional 
is pointwise positive along $\Sigma$, and so $\Sigma$ is strictly stable in the usual sense. 

Tsai and Wang show that there are many examples of strongly stable minimal submanifolds, see \cite[Proposition A]{TsaiWang17}.  Moreover, they show that strong stability implies local uniqueness of $\Sigma$ as a minimal submanifold as follows.

\begin{theorem}[Theorem A, \cite{TsaiWang17}] \label{thm:tsaiwang1} Let $\Sigma^n\subset (M,g)$ be a compact, oriented minimal submanifold which is strongly stable. There exists a tubular neighbourhood $U$ of $\Sigma$ such that $\Sigma$ is the only compact minimal submanifold in $U$ of dimension at least $n$.
\end{theorem}
A further consequence is a dynamical stability result.

\begin{theorem}[Theorem B, \cite{TsaiWang17}]\label{thm:tsaiwang2} Let $\Sigma^n\subset (M,g)$ be a compact, oriented minimal submanifold which is strongly stable. If $\Gamma$ is an $n$-dimensional submanifold that is close to $\Sigma$ in $C^1$, then the mean curvature flow $\Gamma_t$ with $\Gamma_0=\Gamma$ exists for all time, and $\Gamma_t$ converges to $\Sigma$ smoothly as $t\ra \infty$. 
\end{theorem}

We first note the local uniqueness result extends to a considerable weaker setting.

\begin{theorem}\label{thm:mainthm.2} Let $\Sigma^n\subset (M,g)$ be a compact minimal submanifold which is strongly stable. There exists a tubular neighbourhood $U$ of $\Sigma$ such that, up to higher multiples of $\Sigma$, there is no other stationary integral varifold with support in $U$ of dimension greater than or equal to $n$.
\end{theorem}

We also show  that dynamical stability  extends to much weaker initial conditions.

\begin{theorem}\label{thm:mainthm} Let $\Sigma^n\subset (M,g)$ be a compact, oriented minimal submanifold which is strongly stable. Then there exists a tubular neighbourhood $U$ of $\Sigma$ such that the following holds. Let $\Gamma$ be an integral $n$-current in $U$ which is in the same homology class (as currents) as $\Sigma$ in $U$ such that $\bfM [\Gamma] < 2 |\Sigma|$. Furthermore, let $\{\mu_t\}_{t\geq 0}$ be an enhanced Brakke flow starting at $\Gamma$. Then $\mu_t$ is non-vanishing for any $t\geq 0$ and for $t \rightarrow \infty$ converges smoothly to $\Sigma$.
\end{theorem}

Here $|\Sigma|$ denotes the volume of $\Sigma$ and $\bfM [\,\cdot\,]$ the mass of a current. For the definition of an enhanced Brakke flow see Theorem \ref{thm:enhanced-flow}.  We shall deduce Theorem \ref{thm:mainthm.2} from Theorem \ref{thm:mainthm}: both are proved in Section \ref{sec:extension}.

\begin{remark} One can drop the assumption that $\Sigma$ is orientable by working with flat chains mod 2 instead of integral currents. Then the same results hold true.
\end{remark}

We shall apply our results to show that we can, in some important cases of interest, flow through the singularities of Lagrangian mean curvature flow which are proved to occur in the groundbreaking work of Neves \cite{NevesFTS}.  We also obtain global long-time existence and smooth convergence of an enhanced Brakke flow starting from weak initial conditions in key examples of complete Ricci-flat manifolds with special holonomy.  See Section \ref{sec:applications} for these applications.

\paragraph{{\bf Acknowledgements}} This research was supported by an HIMR Focused Research Grant and Leverhulme Trust Research Project Grant RPG-2016-174.

\section{Extension to enhanced Brakke flows}\label{sec:extension}
Recall that a family of Radon measures $(\mu_t)_{t\geq 0}$ on $M$  is called an integral $n$-Brakke flow, provided, given any $\phi \in
C_c^2(M;\R^+)$, the following inequality holds for every $t>0$
\begin{equation}\label{brakkeflow}
\bar{D}_t\mu_t(\phi)\leq \int -\phi |\bfH|^2 + \langle\nabla \phi,
\bfH \rangle \, d\mu_t,
\end{equation}
where $\bar{D}_t$ denotes the upper derivative at time $t$, and $\bfH$ is the weak mean curvature vector. We take
the right-hand side to be $-\infty$  if $\mu_t$ is not the mass measure of an integral $n$-varifold which carries a weak mean curvature which is summable in $L^2$. Note that in the case $\mu_t$ corresponds to a smooth motion by
mean curvature flow, $\bar{D}_t$ is just the usual derivative and we have
equality in \eqref{brakkeflow}. For more details we refer the reader to \cite{Ilmanen}.

We recall Ilmanen's existence result for enhanced Brakke flows, which is proven using an elliptic regularisation
scheme.

\begin{theorem}[\cite{Ilmanen}, \S 8.1]\label{thm:enhanced-flow}
  Let $T_0$ be a local
integral $n$-current in $(M^{n+m},g)$ with $\partial T_0 = 0$, 
finite mass ${\bf M}[T_0]<\infty$ and compact support. There exists a local integral
$(n+1)$-current $T$ in $M\times [0,\infty)$ and a family
$\{\mu_t\}_{t\geq 0}$ of Radon measures on $M$ such that
\begin{itemize}
\item[$(i)$] (a) $\partial T = T_0$\\[1ex]
  (b) ${\bf M}[T_B]$, where $T_B = T\res (M\times B),\ B\subset
  [0,\infty)$, is absolutely continuous with respect to $\mathcal{L}^1(B)$.\\[-2ex]
\item[$(ii)$] (a) $\mu_0=\mu_{T_0}, {\bf M}[\mu_t]\leq {\bf M}[\mu_0]$
  for $t>0$.\\[1ex]
(b) $\{\mu_t\}_{t\geq 0}$ is an integral $n$-Brakke flow.\\[-2ex]
\item[$(iii)$] $\mu_t\geq \mu_{\pi_\#(T_t)}$ for each $t\geq 0$, where $T_t$ is
  the slice $\partial(T\res (M^{m+k}\times [t,\infty))$ and $\pi: M\times \R \ra M$ is the projection on the first factor.  
\end{itemize}
\end{theorem}

 Ilmanen calls  $(\{\mu_t\}_{t\geq 0}, T)$ with the above properties an enhanced Brakke motion. We will instead call this an \emph{enhanced Brakke flow}.

Tsai--Wang's local uniqueness and long-time convergence results (Theorems \ref{thm:tsaiwang1}--\ref{thm:tsaiwang2}) hinge on the following estimate for the squared distance function $\psi$ to $\Sigma$, which we reformulate slightly for our purposes. Note, although stated there, orientability of $\Sigma$ is not needed for the proof.

\begin{proposition}[Proposition 4.1, \cite{TsaiWang17}] \label{thm:distest}
Let $\Sigma^n \subset (M,g)$ be a compact minimal submanifold which is strongly stable. There exist positive constants $\eps_1$ and $c_1$, which depend on the geometry of $M$ and $\Sigma$, such that on the tubular neighbourhood $U_{\eps_1}$ of $\Sigma$ we have:
$$ \text{\emph{tr}}_n \nabla^2\psi \geq c_1 \psi\ ,$$
where $\nabla^2\psi$ is the Hessian of $\psi$, and $\text{\emph{tr}}_n$ is the sum of the smallest $n$ eigenvalues.
\end{proposition}

We now show how Proposition \ref{thm:distest} together with White's barrier theorem, Theorem \ref{thm:barrier}, yields the proof of Theorem \ref{thm:mainthm}.

Let $\Sigma^n \subset (M,g)$ be a compact, oriented minimal submanifold which is strongly stable and consider the tubular neighbourhood $U=U_{\eps_1}$ given by Proposition \ref{thm:distest}. Let $\{\mu_t\}_{t\geq 0}$ be an integral $n$-Brakke flow in $(M,g)$ such that $\spt \mu_0 \subset U$. Recall $c_1>0$ given by Proposition \ref{thm:distest} and consider for any $\eps > 0$ the function
\begin{equation}\label{eq.0}
u(p,t) = e^{c_1 t} \psi - \eps t \ .
\end{equation}
Then we see that
$$\frac{\partial u}{\partial t} - \text{tr}_n \nabla^2 u \leq - \eps < 0\, ,$$
and thus by Theorem \ref{thm:barrier}   that
$$ u(x,t) \leq \eps_1^2$$
on $\spt \mu_t$. Letting $\eps \ra 0$ this implies that 
$$ \psi \leq e^{-c_1t} \eps_1^2$$
on $\spt \mu_t$ and thus 
\begin{equation}\label{eq.1}
\spt \mu_t \subset U_{e^{-c_1 t/2} \eps_1}\, .
\end{equation}
\vspace{1ex}
\begin{proof}[Proof of Theorem \ref{thm:mainthm}]
Continuing to use the notation above, by making $\eps_1$ smaller if necessary we can assume that $[\Sigma]_U \neq 0$,
 where $[\Sigma]_U$ denotes the homology class of $\Sigma$ in $U$ with respect to integral currents. This implies that the infimum of $\bfM[S]$, where $S$ represents the same homology class as $\Sigma$ in $U$, is positive, i.e.
\begin{equation}\label{eq.3}
\delta:=\inf_{S \in [\Sigma]_U} \bfM[S] >0\ .
\end{equation}
Consider now an enhanced Brakke flow $(\{\mu_t\}_{t\geq 0}, T)$ starting at $\Gamma$ such that $\spt \mu_0 \subset U$. By \eqref{eq.1} and by Theorem \ref{thm:enhanced-flow} $(iii)$  we have that
 $$\spt \mu_{\pi_\#(T_t)}  \subset U_{e^{-c_1 t/2} \eps_1 }$$
 and thus $\spt T \subset U \times [0,\infty)$. Since $\partial (T_{[0,t]}) = \Gamma - T_t$ we obain
 $$ \partial \pi_\#(T_{[0,t]}) = \Gamma - \pi_\#(T_t) $$
 and thus $\pi_\#(T_t) \in [\Sigma]_U$ for all $t\geq 0$. By \eqref{eq.3} and Theorem \ref{thm:enhanced-flow} $(iii)$ we obtain
 \begin{equation}\label{eq.4}
 \mu_t(M)\geq \bfM(\pi_\#(T_t))  \geq \delta >0 
 \end{equation}
 for all $t\geq 0$, and thus the flow is non-vanishing. Observe that the definition of  Brakke flow implies that that for any $0\leq t_1 \leq t_2$ one has the estimate
  \begin{equation}\label{eq.5}
\int_{t_1}^{t_2} \int |\bfH|^2\, d\mu_t\, dt \leq \mu_{t_1}(M) - \mu_{t_2}(M)\, ,
 \end{equation}
 where $\bfH$ is the mean curvature vector.  Combining this with \eqref{eq.4} implies that for any sequence $t_i \ra \infty$ there is a subsequence  $t'_i \ra \infty$ such that the flows $\{\mu_{t+t'_i}\}_{-t_i\leq t <\infty}$ converge to a non-vanishing Brakke flow $\{\bar{\mu}_t\}_{t\in \R}$. By \eqref{eq.1} we have $\spt \bar{\mu}_t \subset \Sigma$ for all $t \in \R$ and by \eqref{eq.5} we have that $\bar{\mu}_t$ is the mass measure of a stationary varifold for almost all $t\in \R$. Thus by the constancy theorem, see for example \cite{Simon}, we have for any such $t$ that 
$$ \bar{\mu}_t = \theta \,\mathcal{H}^n\res \Sigma $$
for some constant multiplicity $\theta \in\N$.  By assumption we have $\bfM [\Gamma] < 2 |\Sigma|$ and thus the monotonicity of total measure for Brakke flows implies that the multiplicity $\theta$ has to be one. Thus $\{\bar{\mu}_t\}_{t \in \R}$ is the static Brakke flow corresponding to $\Sigma$. Brakke's regularity theorem, see \cite{Brakke} or \cite{Tonegawa14a, Tonegawa14b}, now implies that the convergence is smooth. This implies that as $t\ra \infty$ the Brakke flow $\{\mu_t\}_{t\geq 0}$ converges smoothly to $\Sigma$. 
\end{proof}
\vspace{1ex}
\begin{proof}[Proof of Theorem \ref{thm:mainthm.2}]
One can use Proposition \ref{thm:distest} and the first variation formula for stationary varifolds to deduce Theorem \ref{thm:mainthm.2}. For convenience we use Theorem \ref{thm:mainthm}. We choose $U = U_{\eps_1}$ as above. Assume $\Gamma^{n+k}$ is a stationary integral varifold with $\spt \Gamma \subset U$. Note first that the barrier \eqref{eq.0} works for all Brakke flows of dimension $n+k \geq n$. We can thus treat $\Gamma^{n+k}$ as a stationary Brakke flow. The proof of Theorem \ref{thm:mainthm} yields that $\spt \Gamma \subset \Sigma$. Thus $k = 0$ and even more $\Gamma$ is the varifold associated to $\Sigma$ up to a constant multiplicity.
\end{proof}

\section{Applications}\label{sec:applications}

\subsection{Singularities of Lagrangian mean curvature flow}
Consider a compact special Lagrangian  $L$ in a Calabi--Yau manifold.  Suppose that $L$ is strongly stable.  For example, we could assume $L$ has positive Ricci curvature, such as the zero section in $T^*S^n$ with the Stenzel metric \cite{Stenzel}, since $L$ is then strongly stable by \cite[Proposition A]{TsaiWang17}: this is a consequence of the Gauss equation and the special Lagrangian condition, which in particular imposes symmetries on the second fundamental form of $L$.  It is worth noting that, by the work of Hein--Sun  \cite{HeinSun}, special Lagrangian $n$-spheres with positive Ricci curvature are now known to exist in certain compact Calabi--Yau $n$-folds.

Using the work of Neves in \cite{NevesFTS} we may construct a Lagrangian $L'$ Hamiltonian isotopic to $L$ which is arbitrarily $C^0$ close to $L$ but Lagrangian mean curvature flow $L'_t$ starting at $L'$ will develop a finite-time singularity.  Thus, $L'$ cannot satisfy the conditions of Tsai--Wang's result, Theorem \ref{thm:tsaiwang2}. 

However, we can choose $L'$ so that $\bfM[L']<2|L|$, and $L'$ is homologous to $L$ since it is Hamiltonian isotopic to $L$.  Moreover, we can ensure that $L'$ lies in the tubular neighbourhood $U$ provided by Theorem \ref{thm:mainthm}, as $L'$ is $C^0$ close to $L$.  Hence, applying Theorem \ref{thm:mainthm} gives that the enhanced Brakke flow starting at $L'$ exists for all time and converges smoothly to $L$.  

For all times before the first singular time of $L'_t$, the enhanced Brakke flow will agree with $L'_t$.  Hence the enhanced Brakke flow enables us to flow through the singularity of $L'_t$ and still converge smoothly to the special Lagrangian $L$.  

It would be useful to study this situation further, to see if this sheds light on the problem of long-time existence and converge of Lagrangian mean curvature flow.

\subsection{Non-compact manifolds with special holonomy}

There are several well-known examples of manifolds $M$ with complete Ricci-flat metrics with special holonomy and maximal volume growth, which have the structure of a vector bundle over a compact base:
\begin{itemize}
\item $T^*S^n$ $(n\geq 2)$ and $T^*\mathbb{CP}^n$ (Calabi--Yau, i.e.~holonomy $\mathrm{SU}(n)$, metrics \cite{Calabi,Stenzel});  
\item $\Lambda^2_-T^*S^4$, $\Lambda^2_-T^*\mathbb{CP}^2$ and the spinor bundle of $S^3$ 
(holonomy $\mathrm{G}_2$ metrics \cite{BryantSalamon}); 
\item the negative spinor bundle of $S^4$ 
(holonomy $\mathrm{Spin(7)}$ metric \cite{BryantSalamon}).
\end{itemize}
In each case the zero section $\Sigma^n$ of the bundle is volume-minimizing (since it is calibrated) and strongly stable by \cite[Proposition A]{TsaiWang17}.  Moreover, the 
squared distance function to $\Sigma$ is strictly convex away from $0$ \cite{TsaiWang16}, so we can take $U=M$ in our Theorems \ref{thm:mainthm.2}--\ref{thm:mainthm} in all of these cases.

We deduce that we get global uniqueness of $\Sigma$ amongst stationary integral varifolds in $M$ with support of dimension at least $n$, up to multiplicity, and long-time smooth convergence to $\Sigma$ of an enhanced Brakke flow starting at any $\Gamma\in[\Sigma]$ with mass strictly less than twice the volume of $\Sigma$.  

Notice in particular in the Calabi--Yau cases that we do not have to start with a Lagrangian and yet we still get convergence of an enhanced Brakke flow to the special Lagrangian base.  As the results of Neves indicate \cite{NevesFTS}, one expects singularities to develop along the flow, even starting with a smooth Lagrangian initial condition, and so the enhanced Brakke flow gives a flow through singularities in these cases.  We would similarly expect the mean curvature flow in the $\mathrm{G}_2$ and $\mathrm{Spin}(7)$ cases to develop singularities in general, yet we can still obtain a flow through singularities to the volume-minimising base.

\begin{appendix}

\section{Avoidance principle in higher codimension.}

We recall White's barrier theorem for mean curvature flow, see \cite[Theorem 14.1]{White_mcfnotes}. We include the proof for completeness.

\begin{theorem}[White]\label{thm:barrier}
 Suppose $\cm$ is the space-time support of an $n$-dimensional integral Brakke flow $\{\mu_t\}_{t\in I}$ in $\Omega \subset M$. Let $u:\Omega \times \R \ra \R$ be a smooth function, so that at $(x_0,t_0)$,
 $$\frac{\partial u}{\partial t} < \text{\emph{tr}}_n \nabla^2u\, ,$$
 where $\nabla^2u$ is the spatial ambient Hessian, and $\text{\emph{tr}}_n$ is the sum of the smallest $n$ eigenvalues. Then
 $$ u\big|_{\cm \cap \{t\leq t_0\}}$$
 cannot have a local maximum at $(x_0,t_0)$.
\end{theorem}
\begin{proof} Suppose otherwise, for a contradiction. We may assume  $\cm = \cm \cap \{t\leq t_0\}$ and that $u|_\cm$ has a strict local maximum at $(x_0,t_0)$.  (Otherwise we could replace $u$ by $u- d(x,x_0)^4 - |t_0-t|^2$).

Let $P(r) = B_r(x_0) \times (t_0-r^2,t_0]$. Choose $r>0$ small enough so that $t_0-r^2$ is past the initial time of the flow, $r$ is smaller than the injectivity radius at $x_0$,  $u|_{\cm\cap \overline{P(r)}} $ has a maximum at $(x_0,t_0)$ and nowhere else and $\tfrac{\partial u}{\partial t} < \text{tr}_n \nabla^2u$ on $\overline{P(r)}$. By adding a constant we can furthermore assume that $u_{\cm \cap (\bar{P}\setminus P)}<0< u(x_0,t_0)$. We let $u^+:= \max\{u,0\}$ and insert $(u^+)^4$ into the definition of Brakke flow. Thus
\begin{equation*}
 \begin{split}
 0&\leq \int_{B_r} (u^+)^4 \, d\mu_{t_0} = \int_{B_r} (u^+)^4 \, d\mu_{t_0} - \int_{B_r} (u^+)^4 \, d\mu_{t_0-r^2}\\
 &\leq \int_{t_0-r^2}^{t_0} \int \bigg(\frac{\partial}{\partial t} (u^+)^4 + \langle \mathbf{H}, \nabla(u^+)^4\rangle - |\mathbf{H}|^2 (u^+)^4 \bigg) \, d\mu_t dt\\
 &\leq \int_{t_0-r^2}^{t_0} \int \bigg(\frac{\partial}{\partial t} (u^+)^4 - \text{div}_{\cm}\big(\nabla(u^+)^4\big) \bigg) \, d\mu_t dt\\
 &= \int_{t_0-r^2}^{t_0} \int 4 \bigg((u^+)^3\frac{\partial}{\partial t} u^+-3 (u^+)^2 |\nabla^\cm u^+|^2 - (u^+)^3\text{div}_{\cm}\big(\nabla(u^+)\big) \bigg) \, d\mu_t dt\\
 &\leq  \int_{t_0-r^2}^{t_0} \int 4 (u^+)^3\bigg(\frac{\partial}{\partial t} u^+ - \text{tr}_{n} \nabla^2u^+ \bigg) \, d\mu_t dt < 0\, ,
\end{split}
\end{equation*}
which is a contradiction.
\end{proof}

\end{appendix}

\providecommand{\bysame}{\leavevmode\hbox to3em{\hrulefill}\thinspace}
\providecommand{\MR}{\relax\ifhmode\unskip\space\fi MR }
\providecommand{\MRhref}[2]{%
  \href{http://www.ams.org/mathscinet-getitem?mr=#1}{#2}
}
\providecommand{\href}[2]{#2}

\end{document}